\documentclass{amsart}
\usepackage{amssymb,mathtools}
\usepackage[british]{babel}
\usepackage{enumitem}
\usepackage[latin1]{inputenc}

\usepackage{url}
\usepackage{tikz-cd}

\newtheorem{theorem}{Theorem}
\newtheorem{corollary}[theorem]{Corollary}
\newtheorem{lemma}[theorem]{Lemma}
\newtheorem{proposition}[theorem]{Proposition}
\theoremstyle{definition}
\newtheorem{definition}[theorem]{Definition}

\numberwithin{equation}{section}

\newcommand{\tl}{\vartriangleleft}
\newcommand{\seq}{\mathsf{Seq}}
\newcommand{\rca}{\mathsf{RCA}}

\title[Higman's lemma is stronger for better quasi orders]{Higman's lemma is stronger\\ for better quasi orders}
\author{Anton Freund}

\address{Anton Freund, Department of Mathematics, Technical University of Darmstadt, Schloss\-garten\-str.~7, 64289~Darmstadt, Germany}
\email{freund@mathematik.tu-darmstadt.de}

\thanks{Funded by the Deutsche Forschungsgemeinschaft (DFG, German Research Foundation) -- Project number 460597863.}

\begin{document}

\begin{abstract}
We prove that Higman's lemma is strictly stronger for better quasi orders than for well quasi orders, within the framework of reverse mathematics. In fact, we show a stronger result: the infinite Ramsey theorem (for tuples of all lengths) follows from the statement that any array $[\mathbb N]^{n+1}\to\mathbb N^n\times X$ for a well order~$X$ and $n\in\mathbb N$ is good, over the base theory~$\mathsf{RCA_0}$.
\end{abstract}


\maketitle

\section{Introduction}

Let $\seq(Q)$ denote the collection of finite sequences in~$Q$. To refer to the entries and lengths of sequences, we stipulate that $\sigma\in\seq(Q)$ is equal to $\langle\sigma_0,\ldots,\sigma_{l(\sigma)-1}\rangle$. Where the context suggests it, we identify $n\in\mathbb N$ with $\{0,\ldots,n-1\}$. When~$Q$ is a quasi order, we define $Q^{<\omega}$ as the quasi order with underlying set $\seq(Q)$ and
\begin{equation*}
\sigma\leq\tau\quad\Leftrightarrow\quad\begin{cases}
\text{there is a strictly increasing $f:l(\sigma)\to l(\tau)$}\\
\text{with $\sigma_i\leq_Q\tau_{f(i)}$ for all $i<l(\sigma)$.}
\end{cases}
\end{equation*}
Higman's lemma~\cite{higman52} is the statement that $Q\mapsto Q^{<\omega}$ preserves well quasi orders.

As finite sequences in~$Q$ correspond to functions~$n\to Q$, a natural generalization leads to transfinite sequences with ordinal numbers as lengths. The collection of transfinite sequences in~$Q$ need not be a well quasi order when~$Q$ is one (see the counterexample due to R.~Rado~\cite{rado-counterexample}). To secure closure properties under infinitary constructions, C.~Nash-Williams has introduced the more restrictive notion of better quasi order~\cite{nash-williams-bqo}. We refer to~\cite{marcone-survey-old} for an introduction that uses the same notation as the present paper. Parts of the definition will also be recalled in the next section. By Nash-Williams' theorem we shall mean the statement that the collection of transfinite sequences in~$Q$ is a better quasi order whenever the same holds for~$Q$ (which is proved in~\cite{nash-williams-bqo}). The statement that $Q\mapsto Q^{<\omega}$ preserves better quasi orders is known as the generalized Higman lemma.

Reverse mathematics is a research program in logic, which aims to determine the minimal axioms that are needed to prove given theorems from various areas of mathematics (see the paper by H.~Friedman~\cite{friedman-rm} and the textbook by S.~Simpson~\cite{simpson09}). A classical result states that Higman's lemma is equivalent to an abstract set existence principle known as arithmetical comprehension, over the weak base theory~$\rca_0$ (see~\cite[Theorem~X.3.22]{simpson09}). Question~24 from a well-known list of A.~Montalb\'an~\cite{montalban-open-problems} asks about the precise strength of Nash-Williams' theorem. The latter is known to imply the principle of arithmetical transfinite recursion (which is considerably stronger than arithmetical comprehension), by a result of R.~Shore~\cite{shore-comp-wos} (see also~\cite{marcone-bad-sequence}).

Conversely, A.~Marcone~\cite{marcone-bad-sequence} has shown that arithmetical transfinite recursion suffices to reduce Nash-Williams' theorem to the generalized Higman lemma. It~remains open whether the latter can be proved by arithmetical transfinite recursion. A proof in $\mathsf{ACA}_0$ (the extension of $\rca_0$ by arithmetical comprehension) had been suggested by P.~Clote~\cite{clote-fraisse}, but according to Marcone it could not be substantiated (see the paragraph after Conjecture~5.6 in~\cite{marcone-bad-sequence}). In the present paper, we show that no such proof can exist: Arithmetical comprehension is known to be strictly weaker than the infinite Ramsey theorem for tuples of all lengths. We will prove that the latter follows from the statement that $\mathbb N^n\times X$ is a better quasi order for any well order~$X$ and all $n\in\mathbb N$ (even when only the barrier $[\mathbb N]^{n+1}$ is considered). This statement is a consequence of the generalized Higman lemma.

The idea of our proof is to iterate an argument due to Marcone, which shows that arithmetical comprehension follows when the better quasi orders are closed under binary products (see Theorem~5.10 and Lemma~5.17 of~\cite{marcone-survey-old}). We do not obtain a new upper bound on the strength of the generalized Higman lemma. The significance of our lower bound is heightened by the fact that many fundamental questions in the reverse mathematics of better quasi orders are wide open~\cite{marcone-survey-new,montalban-open-problems}.

\section{Well foundedness proofs via better quasi orders}

To connect the generalized Higman lemma and the infinite Ramsey theorem, we will use the following transformations of linear orders. The definition employs notation for sequences that is explained at the beginning of the previous section.

\begin{definition}
For a linear order~$X$ and finite sequences $\alpha,\beta\in\seq(X)$, we put
\begin{equation*}
j(\alpha,\beta):=\min\left(\left\{\left.j<\min\big(l(\alpha),l(\beta)\big)\,\right|\,\alpha_j\neq\beta_j\right\}\cup\left\{\min\big(l(\alpha),l(\beta)\big)\right\}\right).
\end{equation*}
On~$\seq(X)$ we consider the lexicographic comparisons given by
\begin{equation*}
\alpha\prec\beta\quad\Leftrightarrow\quad\begin{cases}
\text{either }j:=j(\alpha,\beta)<\min\big(l(\alpha),l(\beta)\big)\text{ and }\alpha_j<_X\beta_j,\\
\text{or }j(\alpha,\beta)=l(\alpha)<l(\beta).
\end{cases}
\end{equation*}
Let $\omega(X)$ be the linear order with lexicographic comparisons and underlying set
\begin{equation*}
\omega(X):=\{\alpha\in\seq(X)\,|\,\alpha_{l(\alpha)-1}\leq_X\ldots\leq_X\alpha_0\}.
\end{equation*}
Finally, we define iterations by stipulating $\omega^X_0:=X$ and $\omega^X_{n+1}:=\omega\left(\omega^X_n\right)$.
\end{definition}

It may help to think of $\alpha\in\omega(X)$ as the Cantor normal form~$\omega^{\alpha_0}+\ldots+\omega^{\alpha_{l(\alpha)-1}}$. Within $\rca_0$, the official definition of~$\omega^X_n$ does not proceed by recursion on~$n\in\mathbb N$. Instead, one first observes that iterated applications of~$\seq$ yield `balanced' trees or terms with leaf labels or constant symbols from~$X$. The relation $\prec$ between trees of height~$n$ and the set of trees in~$\omega^X_n$ can then be determined by primitive recursion over the number of vertices, for all~$n\in\mathbb N$ simultaneously. The following are equivalent over the base theory~$\rca_0$, by results of A.~Marcone and A.~Montalb\'an~\cite{marcone-montalban} as well as C.~Jockusch~\cite{jockusch-ramsey} and K.~McAloon~\cite{mcaloon-ramsey}:
\begin{enumerate}[label=(\roman*)]
\item if $X$ is a well order, then so is $\omega^X_n$ for every~$n\in\mathbb N$,
\item for all~$n\in\mathbb N$, the $n$-th Turing jump of any set exists,
\item the infinite Ramsey theorem holds for tuples of any length.
\end{enumerate}
It is straightforward to conclude that each of these statements is strictly stronger than arithmetical comprehension (see, e.\,g.,~\cite[Section~4]{rathjen-afshari-ramsey}). We will show that~(i) follows from (a weaker statement than) the generalized Higman lemma.

Given a linear order~$Y$, we write $\omega+Y$ for the linear order with underlying set
\begin{equation*}
\omega+Y:=\{(0,n)\,|\,n\in\mathbb N\}\cup\{(1,y)\,|\,y\in Y\}
\end{equation*}
and $(0,n)\leq(0,n')<(1,y)\leq(1,y')$ for $n\leq n'$ in~$\mathbb N$ and $y\leq y'$ in~$Y$. When~$X$ is (isomorphic to) an order of the form $\omega+Y$, we have an element $0:=(0,0)\in X$ and a strictly increasing map $X\ni x\mapsto 1+x\in X$ that is given by $1+(0,n):=(0,1+n)$ and $1+(1,y):=(1,y)$, which yields $0<1+x$ for any~$x\in X$. The next definition and lemma provide a convenient characterization of the order from above.

\begin{definition}
Assume~$X$ has the form $\omega+Y$. For $\alpha\in\omega(X)$ and $j\in\mathbb N$, we put
\begin{equation*}
\overline\alpha_j:=\begin{cases}
1+\alpha_j & \text{if }j<l(\alpha),\\
0 & \text{otherwise}.
\end{cases}
\end{equation*}
Given $\alpha,\beta\in\omega(X)$, we then set $c(\alpha,\beta):=\overline\alpha_j\in X$ with $j:=j(\alpha,\beta)$.
\end{definition}

Let us record the following basic facts.

\begin{lemma}\label{lem:ineq}
Consider $\alpha,\beta,\gamma\in\omega(X)$ with $X$ of the form~$\omega+Y$. We have
\begin{equation*}
\alpha\prec\beta\quad\Leftrightarrow\quad\overline\alpha_j<_X \overline\beta_j\text{ with }j:=j(\alpha,\beta).
\end{equation*}
When we have $\alpha\succ\beta$ and $j(\alpha,\beta)\leq j(\beta,\gamma)$, we get $c(\alpha,\beta)>c(\beta,\gamma)$ in~$X$.
\end{lemma}
\begin{proof}
The equivalence is checked by a case distinction between strict inequalities and equalities in $j(\alpha,\beta)\leq l(\alpha)$ and $j(\alpha,\beta)\leq l(\beta)$. To verify the remaining claim, put $i:=j(\alpha,\beta)=j(\beta,\alpha)$ and $j:=j(\beta,\gamma)$. Given $\alpha\succ\beta$, we get $c(\alpha,\beta)=\overline\alpha_i>\overline\beta_i$ by the equivalence. Due to $\beta\in\omega(X)$ and $i\leq j$, we also have $\overline\beta_i\geq\overline\beta_j=c(\beta,\gamma)$.
\end{proof}

When~$X$ has the form~$\omega+Y$, so has $\omega^X_n$ for all numbers $n\in\mathbb N$. To confirm this for~$n=m+1$, we note that $\omega^X_m$ contains a minimal element, which we denote by~$0$. If we have $m=0$ and thus $\omega^X_m=X$, this element is given as above, while $m=k+1$ leads to $0=\langle\rangle\in\omega(\omega^X_k)=\omega^X_m$. Now the elements $\langle 0,\ldots,0\rangle\in\omega(\omega^X_m)=\omega^X_n$ form an initial segment isomorphic to~$\mathbb N$. We can conclude that the previous considerations apply with~$\omega^X_n$ at the place of~$X$. In particular, we obtain elements $j(\alpha,\beta)\in\mathbb N$ and $c(\alpha,\beta)\in\omega^X_n$ for any $\alpha,\beta\in\omega^X_{n+1}=\omega(\omega^X_n)$.

Let $[\mathbb N]^n$ be the set of strictly increasing sequences $s\in\seq(\mathbb N)$ of length $l(s)=n$. Whenever we use this notation, we assume $n>0$. For $s,t\in[\mathbb N]^n$ we declare
\begin{equation*}
s\tl t\quad:\Leftrightarrow\quad s_0<t_0\text{ and }s_{i+1}=t_i\text{ for all }i<n-1.
\end{equation*}
If we have $n>1$, the second conjunct on the right does already entail $s_0<s_1=t_0$. For $n=1$, the condition $s_0<t_0$ allows us to identify $([\mathbb N]^1,\tl)$ with the isomorphic structure~$(\mathbb N,<)$. Given a quasi order~$Q$, a map $f:[\mathbb N]^n\to Q$ is called good if there are $s\tl t$ with $f(s)\leq_Q f(t)$. Otherwise it is called bad. The structures $([\mathbb N]^n,\tl)$ are examples for the notion of barrier that appears in the definition of better quasi orders (see, e.\,g.,~\cite{marcone-survey-old}). To follow the present paper, it suffices to know that if $Q$ is a better quasi order, then any $f:[\mathbb N]^n\to Q$ is good. We note that $Q$ is a well quasi order precisely when this holds for~$n=1$. Over $\rca_0$, any map $f:[\mathbb N]^n\to Q$ into a well order~$Q$ is good (consider $s^k\tl s^{k+1}$ with $s^k_i:=k+i$ and exploit linearity).

Given $s\tl t$ in $[\mathbb N]^n$, it is standard to define $s\cup t\in[\mathbb N]^{n+1}$ as the sequence with
\begin{equation*}
(s\cup t)_i:=\begin{cases}
s_0 & \text{when }i=0,\\
s_i=t_{i-1} & \text{when }0<i<n,\\
t_{n-1} & \text{when }i=n.
\end{cases}
\end{equation*}
Note that $s\cup t$ is strictly increasing, which relies on~$s_0<t_0$ when we have $n=1$. Any element of $[\mathbb N]^{n+1}$ can be uniquely written as $s\cup t$ with $s\tl t$ in~$[\mathbb N]^n$. When~we use the notation $s\cup t$, we always assume $s\tl t$. Let us observe that $r\cup s\tl s'\cup t$ in~$[\mathbb N]^{n+1}$ entails that we have~$s=s'$.

To show that there can be no strictly decreasing sequence $f:\mathbb N\to\omega^X_n$, we now construct maps~$f_k$ with increasingly complex domain but ever simpler codomain. For $\sigma=\langle\sigma_0,\ldots,\sigma_{k-1}\rangle\in\mathbb N^k$ and $j\in\mathbb N$ we write $\sigma^\frown j:=\langle\sigma_0,\ldots,\sigma_{k-1},j\rangle\in\mathbb N^{k+1}$.

\begin{definition}\label{def:f_k}
Consider $f:\mathbb N\to\omega^X_n$ for~$X$ of the form~$\omega+Y$. To define
\begin{equation*}
f_k:[\mathbb N]^{k+1}\to\mathbb N^k\times\omega^X_{n-k}
\end{equation*}
by recursion on~$k\leq n$, we stipulate $f_k(r):=\langle f^0_k(r),f^1_k(r)\rangle$ with
\begin{align*}
f^0_0(\langle i\rangle)&:=\langle\rangle\in\mathbb N^0,\qquad &f^0_{k+1}(s\cup t)&:=f^0_k(s)\,^\frown\,j\big(f^1_k(s),f^1_k(t)\big),\\
f^1_0(\langle i\rangle)&:=f(i),\qquad &f^1_{k+1}(s\cup t)&:=c\big(f^1_k(s),f^1_k(t)\big).
\end{align*}
\end{definition}

The family of functions~$f_k$ can be seen as a single function on sequences, as $k$ is determined by the length of the argument. This single function can be constructed by a recursion over subsequences, which is available in the base theory~$\rca_0$.

Let us declare that $(\sigma_0,\ldots,\sigma_{k-1},\alpha)\leq(\tau_0,\ldots,\tau_{k-1},\beta)$ holds in~$\mathbb N^k\times\omega^X_{n-k}$ if we have $\alpha\preceq\beta$ in $\omega^X_{n-k}$ as well as $\sigma_i\leq\tau_i$ in~$\mathbb N$ for all~$i<k$. The following observation is the crucial step in our argument.

\begin{proposition}
If $f_k$ is bad, then so is $f_{k+1}$ (in the situation of Definition~\ref{def:f_k}).
\end{proposition}
\begin{proof}
Aiming at a contradiction, we assume
\begin{equation}\label{eq:contradict}\tag{$\star$}
f_{k+1}(r\cup s)\leq f_{k+1}(s\cup t)
\end{equation}
with $r\cup s\tl s\cup t$ in $[\mathbb N]^{k+2}$. From~(\ref{eq:contradict}) we can, first, conclude that $f^0_k(r)\leq f^0_k(s)$ holds in $\mathbb N^k$ (i.\,e., componentwise). Given that $f_k$ is bad and that we have $r\tl s$, we must thus have $f^1_k(r)\succ f^1_k(s)$ in $\omega^X_{n-k}=\omega(\omega^X_{n-(k+1)})$. Now~(\ref{eq:contradict}) does, secondly, entail~$j(f^1_k(r),f^1_k(s))\leq j(f^1_k(s),f^1_k(t))$. We obtain $c(f^1_k(r),f^1_k(s))\succ c(f^1_k(s),f^1_k(t))$ due to Lemma~\ref{lem:ineq}. But by~(\ref{eq:contradict}) we do, finally, get the converse inequality as well.
\end{proof}

In the following result, one may take~$X=Z$ when~$Z$ itself has an initial segment that is isomorphic to~$\mathbb N$. Otherwise, we can always put $X:=\omega+Z$, which is a well order whenever the same holds for~$Z$, provably in~$\rca_0$.

\begin{theorem}[{$\rca_0$}]
Consider a linear order~$Z$ with an embedding into a linear order~$X$ that has the form~$\omega+Y$. If all maps $[\mathbb N]^{n+1}\to\mathbb N^n\times X$ are good, then $\omega^Z_n$ is a well order. In particular, this follows when $\mathbb N^n\times X$ is a better quasi order.
\end{theorem}
\begin{proof}
One readily constructs an embedding of~$\omega^Z_n$ into~$\omega^X_n$. So it suffices to show that the latter is a well order. Towards a contradiction, we assume that $f:\mathbb N\to\omega^X_n$ is strictly decreasing. Let $f_k$ for~$k\leq n$ be given as in Definition~\ref{def:f_k}. The map~$f_0$ is bad by our assumption on~$f$. In view of the previous proposition, we can use induction to conclude that $f_n:[\mathbb N]^{n+1}\to\mathbb N^n\times X$ is bad, against the assumption of the theorem. Concerning formalization in~$\rca_0$, we note that the induction~statement is~$\Pi^0_1$ (and an even simpler induction over subsequences would also be possible).
\end{proof}

Finally, we draw the promised conclusions:

\begin{corollary}
Over~$\rca_0$, the infinite Ramsey theorem (for tuples of all lengths) follows from the statement that $\mathbb N^n\times X$ is a better quasi order for every well order~$X$ and all~$n\in\mathbb N$. In particular, it follows from the generalized Higman lemma, so that the latter cannot be proved in $\mathsf{ACA}_0$.
\end{corollary}
\begin{proof}
The first claim follows by the discussion at the beginning of this section. To reduce to the generalized Higman lemma, consider an arbitrary well order~$X$. Let~$Y$ be the order~$\omega+X$ or some other well order into which $\mathbb N$ and $X$ embed. We note that $Y$ is a better quasi order, provably in~$\rca_0$ (see above or~\cite[Lemma~3.1]{marcone-survey-old}). By the generalized Higman lemma, it follows that $Y^{<\omega}$ is a better quasi order. But the latter embeds $Y^{n+1}$ and hence $\mathbb N^n\times X$ for any~$n\in\mathbb N$.
\end{proof}

\bibliographystyle{amsplain}
\bibliography{Generalized-Higman}

\end{document}